\documentclass[11pt]{article}
\usepackage{amsmath, amssymb, amsthm}
\usepackage{verbatim}
\usepackage{multicol}
\usepackage{enumerate}
\usepackage{comment}
\usepackage[none]{hyphenat}
\usepackage{hyperref}
\hypersetup{
	colorlinks=true,
	linkcolor=blue,
	filecolor=magenta,
	urlcolor=cyan,
	citecolor=blue
}
\usepackage{graphicx}
\usepackage{tikz}
\usetikzlibrary{positioning,arrows,shapes,decorations.markings,decorations.pathreplacing,matrix,patterns}
\tikzstyle{vertex}=[circle,draw=black,fill=black,inner sep=0,minimum size=3pt,text=white,font=\footnotesize]

\date{}
\title{\vspace{-0.8cm}Erd\H os-Hajnal-type results for ordered paths}
\author{
	J\'{a}nos Pach \thanks{R\'enyi Institute, Budapest, \emph{e-mail}: \textbf{pach@cims.nyu.edu}} \thanks{IST Austria, Vienna, partially supported by Austrian Science Fund (FWF), grant Z 342-N31.}
\thanks{MIPT, Moscow, partially supported by the Ministry of Education and Science of the Russian Federation in the framework of MegaGrant no 075-15-2019-1926.}
	\and
	Istv\'{a}n Tomon \footnotemark[3] \thanks{ETH Zurich,  Research partially supported by Swiss National Science Foundation grants no. 200020-175573. \emph{e-mail}: \textbf{istvan.tomon@math.ethz.ch}}	
}

\theoremstyle{plain}
\newtheorem{theorem}{Theorem}
\newtheorem{definition}[theorem]{Definition}

\newtheorem{claim}[theorem]{Claim}
\newtheorem{lemma}[theorem]{Lemma}

\theoremstyle{definition}

\newcommand{\subs}{\subseteq}

\begin{document}

\maketitle
\sloppy

\begin{abstract}
An {\em ordered graph} is a graph with a linear ordering on its vertex set. We prove that for every positive integer $k$, there exists a constant $c_k>0$ such that any ordered graph $G$ on $n$ vertices with the property that neither $G$ nor its complement contains an induced monotone path of size $k$, has either a clique or an independent set of size at least $n^{c_k}$. This strengthens a result of Bousquet, Lagoutte, and Thomass\'e, who proved the analogous result for unordered graphs. 	
\smallskip

A key idea of the above paper was to show that any unordered graph on $n$ vertices that does not contain an induced path of size $k$, and whose maximum degree is at most $c(k)n$ for some small $c(k)>0$, contains two disjoint linear size subsets with no edge between them. This approach fails for ordered graphs, because the analogous statement is false for $k\geq 3$, by a construction of Fox. We provide further examples how this statement fails for ordered graphs avoiding other ordered trees as well.
\end{abstract}

\section{Introduction}

Erd\H{o}s and Hajnal \cite{EH89} proved that graphs avoiding some fixed induced subgraph or subgraphs have very favorable Ramsey-theoretic properties. In particular, they contain surprisingly large homogeneous (that is, complete or empty) subgraphs and bipartite subgraphs. According to the celebrated Erd\H{o}s-Hajnal conjecture, every graph $G$ on $n$ vertices which does not contain some fixed graph $H$ as an induced subgraph, has a clique or an independent set of size at least $n^{c}$, where $c=c(H)>0$ is a constant that depends only on $H$. There is a rapidly growing body of literature studying this conjecture (see, e.g., \cite{APPRS05, APS01, BLT15, C14, CSSS18, EHP00, FP08, FS09,SSS20}).
\smallskip

For any graph $G$ and any disjoint subsets  $A, B\subset V(G)$, we say that {\em $A$ is complete to $B$} if $ab\in E(G)$ for every $a\in A, b\in B$. If $|A|=|B|=k$ and $A$ is complete to $B$, then $A$ and $B$ are said to form a {\em bi-clique of size $k$}. Denote the maximum degree of the vertices in $G$ by $\Delta(G)$. Following~\cite{FP08}, a family of graphs $\mathcal{G}$ is said to have the \emph{Erd\H{o}s-Hajnal} property if there exists a constant $c=c(\mathcal{G})>0$ such that every $G\in\mathcal{G}$ has either a clique or an independent set of size at least $|V(G)|^{c}$. The family $\mathcal{G}$ has the \emph{strong Erd\H{o}s-Hajnal property} if there exists a constant $b=b(\mathcal{G})>0$ such that for every $G\in\mathcal{G}$, either $G$ or its complement $\overline{G}$ has a bi-clique of size $b|V(G)|$. It was proved in \cite{APPRS05} that if a \emph{hereditary} family (that is, a family closed under taking induced subgraphs) has the strong Erd\H{o}s-Hajnal property, then it also has the Erd\H{o}s-Hajnal property.
\smallskip

The aim of this paper is to discuss Erd\H os-Hajnal type problems for ordered graphs. An \emph{ordered graph} is a graph with a total ordering on its vertex set. With a slight abuse of notation, in every ordered graph, we denote this ordering by $\prec$. If the vertex set of $G$ is a subset of the integers, then $\prec$ stands for the natural ordering. An ordered graph $H$ is an \emph{ordered subgraph} (or simply subgraph) of $G$ if there exists an order preserving embedding from $V(H)$ to $V(G)$ that maps edges to edges. If, in addition, non-edges are mapped into non-edges, then $H$ is called an {\em induced} ordered subgraph of $G$. If $G$ does not have $H$ as induced ordered subgraph, then we say that $G$ {\em avoids} $H$. The ordered path with vertices $1,\dots,k$ and edges $\{i,i+1\}$, for $i=1,\dots,k-1$, is called a \emph{monotone path of size $k$}.

Our main result is the following.

\begin{theorem}\label{thm:mainthm}
For any positive integer $k$, there exists $c=c(k)>0$ with the following property. If $G$ is an ordered graph on $n$ vertices such that neither $G$ nor its complement contains an induced monotone path of size $k$, then $G$ has either a clique or an independent set of size at least $n^{c}$.
\end{theorem}

Our theorem obviously implies the analogous statement for unordered graphs, which was first established by Bousquet, Lagoutte, and Thomass\'e \cite{BLT15}. The idea of their proof was the following. We call a family of graphs, $\mathcal{H}$,  \emph{lopsided} if there exists a constant $c=c(\mathcal{H})>0$ with the following property: any graph $G$ on $n$ vertices which does not contain any element of $\mathcal{H}$ as an induced subgraph, and for which $\Delta(G)<cn$, the complement of $G$ has a bi-clique of size at least $cn$. If $\mathcal H$ consists of a single graph $H$, then $H$ is called lopsided. They proved that the (unordered) path of size $k$ is lopsided. It follows from the arguments of Bousquet {\em et al.} that if $\mathcal H$ is lopsided, then the family of all graphs which avoid every element of $\mathcal H$ as an induced subgraph, and whose complements also avoid them, has the strong Erd\H os-Hajnal and, thus, the Erd\H os-Hajnal property.

Since then, this idea has been exploited to prove the Erd\H{o}s-Hajnal conjecture for various other families of graphs: the family of graphs avoiding a tree $T$ and its complement~\cite{CSSS18}, the family of graphs avoiding all subdivisions of a graph $H$ and the complements of these subdivisions~\cite{CSSS18+}, the family of graphs avoiding a graph $H$ as a vertex minor~\cite{CO18},  families of graphs avoiding a fixed cycle as a pivot minor \cite{KO20}, etc.
\smallskip

However, for ordered graphs, this method does not work even in the simplest case: for monotone paths. A construction of Fox \cite{F06} shows that, for every $n$ and $\delta>0$, there exists an ordered graph $G$ with $|V(G)|=n$ and $\Delta(G)<n^{\delta}$ which avoids the monotone path of size $3$, and whose complement does not contain a bi-clique of size larger than $\frac{cn}{\log n}$, for a suitable constant $c=c(\delta)>0$. Hence, using the above terminology, the monotone path of size at least 3 is {\em not lopsided}.

Although monotone paths are not lopsided, they satisfy a somewhat weaker property, as is shown by the following theorem of the authors.

\begin{theorem}\label{thm:regi}(\cite{PT19})
For any positive integer $k$, there exists a constant $c=c(k)>0$ with the following property. If $G$ is an ordered graph on $n$ vertices that does not contain an induced monotone path of size $k$, and $\Delta(G)<cn$, then the complement of $G$ contains a bi-clique of size at least $\frac{cn}{\log n}$.
\end{theorem}

Unfortunately, Theorem~\ref{thm:mainthm} cannot be deduced from this weaker property. Our approach is based on a technique in \cite{T20}, where it was shown that the family of string graphs has the Erd\H os-Hajnal property.

Recently, Seymour, Scott, and Spirkl \cite{SSS20} extended our Theorem~\ref{thm:regi} from monotone paths to all ordered forests $T$, albeit with a weaker bound $n^{1-o(1)}$ in place of $\frac{cn}{\log n}$. They proved that for any $0<c<1$, there exists $\epsilon=\epsilon(T,c) > 0$ with the following property. If $G$ is an ordered graph on $n$ vertices that does not contain $T$ as an induced ordered subgraph and $\Delta(G)<\epsilon n$, then the complement of $G$ contains a bi-clique of size at least $\epsilon n^{1-c}$. Therefore, if we want to guarantee a bi-clique of size $n^{1-o(1)}$ in $\overline{G}$, we need to assume that the maximum degree of $G$ is $o(n)$. This is definitely a stronger condition than the one we had for monotone paths.

Our next construction shows that this stronger condition is indeed necessary. We also provide new examples of ordered trees $T$ (that do not contain a monotone path of size 3), for which one cannot expect to find linear size bi-cliques.

\begin{theorem}\label{thm:construction}
	For any $\epsilon>0$ there exist $\delta=\delta(\epsilon)>0$ and $n_0=n_0(\epsilon)$ with the following property.
	
	For any positive integer $n\ge n_0$, there is an ordered graph $G$ with $n$ vertices and $\Delta(G)\leq \epsilon n$ such that the size of the largest bi-clique in $\overline{G}$ is at most $n^{1-\delta}$, and $G$ does not contain either of the following ordered trees as an induced ordered subgraph:
	
	
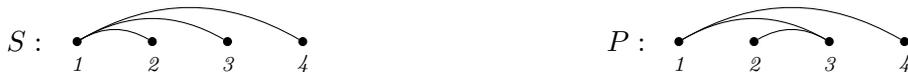
\begin{center}
		\begin{tikzpicture}
		\node at (-0.7,0) {$S:$} ;
		\node[vertex,minimum size=3pt] (A) at (0,0) {} ; \node at (0,-0.3) {\scriptsize 1} ;
		\node[vertex,minimum size=3pt] (B) at (1,0) {} ;\node at (1,-0.3) {\scriptsize 2} ;
		\node[vertex,minimum size=3pt] (C) at (2,0) {} ;\node at (2 ,-0.3) {\scriptsize 3} ;
		\node[vertex,minimum size=3pt] (D) at (3,0) {} ; \node at (3,-0.3) {\scriptsize 4} ;
		
		\draw (A) edge[bend left] (B) ; \draw (A) edge[bend left] (C) ; \draw (A) edge[bend left] (D) ;
		
		\node at (7.3,0) {$P:$} ;
		\node[vertex,minimum size=3pt] (A) at (8,0) {} ; \node at (8,-0.3) {\scriptsize 1} ;
		\node[vertex,minimum size=3pt] (B) at (9,0) {} ;\node at (9,-0.3) {\scriptsize 2} ;
		\node[vertex,minimum size=3pt] (C) at (10,0) {} ;\node at (10 ,-0.3) {\scriptsize 3} ;
		\node[vertex,minimum size=3pt] (D) at (11,0) {} ; \node at (11,-0.3) {\scriptsize 4} ;
		\draw (A) edge[bend left] (D) ; \draw (A) edge[bend left] (C) ; \draw (B) edge[bend left] (C) ;
		\end{tikzpicture}
	\end{center}
\end{theorem}

\medskip

The investigation of bipartite variants of the problems considered in this paper were initiated in~\cite{KPT19}; see also \cite{ATW19, SS20}.

Our paper is organized as follows. In Section~\ref{sect:monpath1}, we introduce the key concept needed for the proof of Theorem~\ref{thm:mainthm} and reduce Theorem~\ref{thm:mainthm} to another statement (Theorem~\ref{thm:qEH}). Sections~\ref{sect:monpath2} and~\ref{sect:monpath3} are devoted to the proof of this latter statement. The construction proving Theorem~\ref{thm:construction} will be presented in Sections~\ref{constructionsection}.
\smallskip

Throughout this paper, we use the following notation, which is mostly conventional. For any graph $G$ and any subset $U\subset V(G)$, we denote by $G[U]$ the subgraph of $G$ induced by $U$. The {\em neighborhood} of $U$ is defined as $N_{G}(U)=N(U)=\{v\in V(G)\setminus U:\exists u\in U, uv\in E(G)\}$. If $U=\{u\}$, instead of $N(U)$, we simply write $N(u)$. For a vertex $v\in V(G)$, let $G-v$ stand for the graph obtained from $G$ by deleting the vertex $v$. Also, if $G$ is an ordered graph, the \emph{forward neighbourhood} of a vertex $v\in V(G)$, denoted by $N^{+}_{G}(y)=N^{+}(y)$ is the set of neighbours $y$ such that $x\prec y$.  

For easier readability, we omit the use of floors and ceilings, whenever they are not crucial.

\section{The quasi-Erd\H{o}s-Hajnal property}\label{sect:monpath1}

After introducing some notation and terminology, we outline our proof strategy for Theorem \ref{thm:mainthm}.

For any $k\ge 3$, let $\mathcal{P}_k$ denote the family of all ordered graphs $G$ such that neither $G$ nor its complement contains a monotone path of size $k$ as an induced subgraph. Instead of proving that $\mathcal{P}_k$ has the Erd\H{o}s-Hajnal property, we prove that it has the \emph{quasi-Erd\H{o}s-Hajnal property}. This concept was introduced by the second named author in~\cite{T20}, in order to show that the family of string graphs has the Erd\H{o}s-Hajnal property.

\begin{definition}\label{elso}
	A family of graphs, $\mathcal{G}$, has the \emph{quasi-Erd\H{o}s-Hajnal property} if there is a constant $c=c(\mathcal G)>0$ with the following property. For every $G\in \mathcal{G}$ with at least 2 vertices, there exist $t\geq 2$ and $t$ disjoint subsets $X_1,\dots,X_t\subset V(G)$ such that $t\geq (\frac{|V(G)|}{|X_i|})^c$ for $i=1,\dots,t$, and 
	\begin{description}
		\item[(i)] either there is no edge between $X_i$ and $X_j$ for $1\leq i<j\leq t$, 
		\item[(ii)] or $X_i$ is complete to $X_j$ for $1\leq i<j\leq t$.
	\end{description}
\end{definition}
\smallskip

It was proved in \cite{T20} that in hereditary families, the quasi-Erd\H{o}s-Hajnal property is equivalent to the Erd\H{o}s-Hajnal property. We somewhat relax the definition of the quasi-Erd\H{o}s-Hajnal property, and with a slight abuse of notation, we overwrite the previous definition as follows.

\begin{definition}\label{masodik}
	A family of graphs, $\mathcal{G}$, has the \emph{quasi-Erd\H{o}s-Hajnal property} if there are two constants, $\alpha,\beta>0,$ with the following property. For every $G\in \mathcal{G}$ with at least 2 vertices, there exist $t\geq 2$ and $t$ disjoint subsets $X_1,\dots,X_t\subset V(G)$ such that $t\geq \alpha(\frac{|V(G)|}{|X_i|})^\beta$ for $i=1,\dots,t$, and 
	\begin{description}
		\item[(i)] either there is no edge between $X_i$ and $X_j$ for $1\leq i<j\leq t$,
		\item[(ii)] or $X_i$ is complete to $X_j$ for $1\leq i<j\leq t$.
	\end{description}
\end{definition}
\smallskip

It is easy to verify that the two definitions are in fact equivalent. If $\mathcal G$ satisfies Definition~\ref{elso}, then, obviously, it also satisfies Definition~\ref{masodik}. In the reverse direction, setting $c=\frac{\beta}{1-\log_2\alpha}$ if $\alpha\leq 1$, and $c=\beta$ if $\alpha>1$, if the inequality  $t\geq \alpha(\frac{|V(G)|}{|X_i|})^\beta$ holds for some $t\ge 2$, then we also have $t\ge (\frac{|V(G)|}{|X_i|})^c$.
\smallskip

Therefore, it is enough to show that $\mathcal{P}_k$ has the quasi-Erd\H{o}s-Hajnal property. The advantage of the quasi-Erd\H{o}s-Hajnal property compared to the Erd\H{o}s-Hajnal property is that it allows us to establish the following lopsided statement, which will imply Theorem \ref{thm:mainthm}.

\begin{theorem}\label{thm:qEH}
For every positive integer $k$, there exist two constants $\epsilon,\alpha>0$ with the  following property.
    
Let $G$ be an ordered graph on $n$ vertices with maximum degree at most $\epsilon n$ such that $G$ does not contain a monotone path of size $k$ as an induced subgraph. Then there exist $t\geq 2$ and $t$ disjoint subsets $X_1,\dots,X_t\subset V(G)$ such that $t\geq \alpha(\frac{n}{|X_i|})^{1/2}$ and there is no edge between $X_i$ and $X_j$ for $1\leq i<j\leq t$.
\end{theorem}

In the inequality $t\geq \alpha(\frac{n}{|X_i|})^{1/2}$, the exponent $1/2$ has no significance: the statement remains true with any $0<\beta<1$ instead of $1/2$ (with the cost of changing $\epsilon$ and $\alpha$). However, it is not true with $\beta=1$, as it would contradict the aforementioned construction of Fox \cite{F06}.
\smallskip

In the rest of this section, we show how Theorem \ref{thm:qEH} implies Theorem \ref{thm:mainthm}. Very similar ideas were used in \cite{BLT15,CSSS18,CSSS18+}. The next two sections are devoted to the proof of Theorem \ref{thm:qEH}.
\smallskip

By a classical result of R\"{o}dl \cite{R86}, any graph $G$ avoiding some fixed graph $H$  contains a linear size subset that is either very dense or very sparse. A quantitatively stronger version of this result was proved by Fox and Sudakov \cite{FS08}.

\begin{lemma}\label{lemma:improveduni}\cite{R86}
	For every graph $H$ and $\epsilon_{0}>0$, there exists $\delta_{0}>0$ with the following property.
	
	For any graph $G$ with $n$ vertices that does not contain $H$ as an induced subgraph, there is a subset $U\subset V(G)$ such that $|U|\geq \delta_{0} n$, and either $|E(G[U])|\leq \epsilon_{0} \binom{|U|}{2}$ or $|E(G[U])|\geq (1-\epsilon_{0}) \binom{|U|}{2}$.
\end{lemma}

Lemma \ref{lemma:improveduni} applies to unordered graphs, but it can be easily extended to ordered graphs, using the following statement.

\begin{lemma} \cite{RW89}\label{lemma:unordered}
	For every ordered graph $H$, there exists an unordered graph $H_{0}$ with the property that introducing any total ordering on $V(H_{0})$, the resulting ordered graph $H_{0}'$ always contains $H$ as an induced ordered subgraph.
\end{lemma}

By the combination of these two lemmas, we obtain the following.

\begin{lemma}\label{lemma:uni2}
	For every ordered graph $H$ and $\epsilon>0$, there exists $\delta>0$ with the following property.
	
  For any ordered graph $G$ with $n$ vertices that does not contain $H$ as an induced ordered subgraph, there exists a subset $U\subset V(G)$ such that $|U|\geq \delta n$, and either $\Delta(G[U])\leq \epsilon|U|$ or $\Delta(\overline{G}[U])\leq \epsilon |U|$.
\end{lemma}

\begin{proof}
	By Lemma \ref{lemma:unordered}, there exists a graph $H_{0}$ such that introducing any total ordering on $V(H_{0})$, the resulting ordered graph $H_{0}'$ contains $H$ as an induced ordered subgraph. Let $\epsilon_0=\frac{\epsilon}{2}$, and let $\delta_0$ be the constant given by Lemma \ref{lemma:improveduni} with respect to $H_{0}$ and $\epsilon_0$.
	
	Let $G$ be an ordered graph with $n$ vertices that does not contain $H$ as an induced ordered subgraph. Then the underlying unordered graph of $G$ does not contain $H_{0}$ as an induced subgraph. Hence, there exists $U'\subset V(G)$ such that $|U'|\geq \delta_0 n$, and either $|E(G[U'])|\leq \epsilon_0 \binom{|U'|}{2}$ or $|E(G[U'])|\geq (1-\epsilon_0) \binom{|U'|}{2}$. Suppose that $|E(G[U'])|\leq \epsilon_0 \binom{|U'|}{2}$, the other case can be handled similarly. Let $W$ be the set of vertices in $U'$ whose degree in $G[U]$ is larger than $2\epsilon_{0} |U|$. Then
	$$\frac{1}{2}(2\epsilon_0|W|)|U'|\leq |E(G[U'])|\leq \epsilon_0 \binom{|U'|}{2},$$
so that $|W|\leq \frac{|U'|}{2}$. Setting $U=U'\setminus W$, we have $\Delta(G[U])\leq 2\epsilon_0 |U'|\leq \epsilon |U|$ and
	$$|U|\geq \frac{|U'|}{2}\geq \frac{\delta_0}{2}n.$$
	Hence, $\delta=\frac{\delta_0}{2}$ will suffice.
\end{proof}

After this preparation, it is easy to deduce from Theorem \ref{thm:qEH} that $\mathcal{P}_k$ has the quasi-Erd\H{o}s-Hajnal property and, therefore, the Erd\H{o}s-Hajnal property.

\begin{proof}[Proof of Theorem \ref{thm:mainthm}]
	Let $\epsilon,\alpha>0$ be the constants given by Theorem \ref{thm:qEH}, and let $\delta>0$ be the constant given by Lemma \ref{lemma:uni2}, where $H$ is the monotone path of size $k$.
	
	Let $G$ be an ordered graph on $n$ vertices such that neither $G$ nor its complement contains a monotone path of length $k$ as an induced subgraph.  Then there exists $U\subset V(G)$ such that $|U|\geq \delta n$, and either $\Delta(G[U])<\epsilon |U|$ or $\Delta(\overline{G}[U])<\epsilon |U|$. Suppose that $\Delta(G[U])<\epsilon |U|$, the other case can be handled similarly. Applying Theorem \ref{thm:qEH} to $G[U]$, we obtain that there exist $t\geq 2$ and $t$ disjoint sets $X_1,\dots,X_t\subset U$ such that
	$$t\geq \alpha\left(\frac{|U|}{|X_i|}\right)^{1/2}\geq \alpha \delta^{1/2}\left(\frac{n}{|X_i|}\right)^{1/2}$$ for $i=1,\dots,t$, and there is no edge between $X_i$ and $X_j$ for $1\leq i<j\leq t$.
	
	Thus, the family $\mathcal{P}_k$ has the quasi-Erd\H{o}s-Hajnal property with parameters $\alpha:=\alpha\delta^{1/2}$ and $\beta:=1/2$. Therefore, $\mathcal{P}_k$ also has the Erd\H{o}s-Hajnal property.
\end{proof}

In the next two sections, we present the proof of Theorem \ref{thm:qEH}.

\section{The embedding lemma}\label{sect:monpath2}

The backbone of the proof of Theorem \ref{thm:qEH} is the following technical lemma, whose proof is already contained in \cite{T20}, within the proof Lemma 7. For convenience and to make this paper self-contained, it is also included here.

\begin{lemma}\label{lemma:embedding}
	There exist two constants $\epsilon_1,\alpha_1>0$ with the following property. Let $G$ be a bipartite graph with vertex classes $A$ and $B$, $|A|=|B|=n$. Then at least one of the following three conditions is satisfied.
	\begin{description}
		\item[(i)] There exist $t\geq 2$ and $2t$ disjoint sets $W_1,\dots,W_t\subset A$ and $X_1,\dots,X_t\subset B$ such that $t\geq \alpha_1(\frac{n}{|X_i|})^{1/2}$, and $X_{i}\subset N(W_{i})$ for $i=1,\dots,t$, but $X_{i}\cap N(W_j)=\emptyset$ for $i\neq j$.
		\item[(ii)] There exist $X_1\subset A$ and $X_2\subset B$ such that $2>\alpha_1 (\frac{n}{|X_i|})^{1/2}$ and there is no edge between $X_1$ and $X_2$.
		\item[(iii)] There exists $v\in A$ such that $|N(v)|\geq \epsilon_1 n.$
	\end{description}
\end{lemma}
\begin{proof}
	We show that $\epsilon_1=\frac{1}{2000}$ and $\alpha_1=\frac{1}{200}$ meet the above requirements.
	
	Suppose that (iii) does not hold. Then the number of edges of $H$ is at most $\epsilon_1 n^2$, so the number of vertices $w\in B$ such that $|N(w)|>\epsilon_1 n$ is at most $n/2$. Deleting all such vertices, and some more, we obtain a bipartite graph $H'$ with vertex classes $A'$ and $B'$ of size $n'=n/2$ such that the maximum degree of $H'$ is at most $2\epsilon_1 n=4\epsilon_1 n'$.
	
	Let $\epsilon=4\epsilon_1=\frac{1}{500}$ and $\alpha=\frac{1}{100}$. From now on, we shall only work with $H'$, so with a slight abuse of notation, write $H:=H'$, $A_{0}:=A'$, $B_{0}:=B'$, and $n:=n'$. Therefore, we have $\Delta(H)\leq \epsilon n$.
	
	In what follows, we describe an algorithm, which will be referred to as the \emph{main algorithm}. It will output 
	\begin{description}
		\item[(i)'] either an integer $t\geq 2$ and $2t$ disjoint sets $W_1,\dots,W_t\subset A$ and $X_1,\dots,X_t\subset B$ such that $t\geq \alpha(\frac{n}{|X_i|})^{1/2}$, and $X_{i}\subset N(W_{i})$ for $i=1,\dots,t$, but $X_{i}\cap N(W_j)=\emptyset$ for $i\neq j$;
		\item[(ii)'] or two subsets $X_1\subset A$ and $X_2\subset B$ such that $2>\alpha (\frac{n}{|X_i|})^{1/2}$ and there is no edge between $X_1$ and $X_2$.
	\end{description}
	
	We declare the following constants for the main algorithm. Let $J_{0}=\lfloor \log_2 \epsilon n\rfloor+1$, and for $j=1,\dots,J_0$, let $t_j=n^{1/2}2^{j/2}$. Then \begin{equation}\label{equ1}
	\sum_{i=1}^{J_0}t_i=\sum_{i=1}^{J_0} n^{1/2} 2^{i/2}\leq 2n\epsilon^{1/2}\frac{1}{1-2^{-1/2}}<\frac{n}{4}.
	\end{equation}
	
	Also, declare the following variables. Let $J:=J_{0}$, $A:=A_{0}$,  $B:=B_{0}$, $A^{*}:=\emptyset$ and $B^{*}:=\emptyset$.

	In each step of the main algorithm, we  make the following changes: we move certain elements of $A$ into $A^{*}$, move certain elements of $B$ into $B^{*}$, and decrease $J$. We think of the elements of $A^{*}$ and $B^{*}$ as ``leftovers''. We make sure that at the end of each step of the algorithm, the following properties are satisfied:
	\begin{enumerate}
		\item	$|A|+|A^{*}|=|B|+|B^{*}|=n$,
		\item  $\displaystyle|A^{*}|,|B^{*}|\leq 2\sum_{i=J+1}^{J_{0}} t_i,$
		\item  for every $v\in B$, $|N(v)\cap A|<2^{J}$.
	\end{enumerate}
	Note that by (\ref{equ1}) and conditions 1 and 2, we have $|A|,|B|\geq \frac{n}{2}$. These conditions are certainly satisfied at the beginning of the algorithm. Next, we describe a general step of our main algorithm.
	
	\bigskip
	
	\noindent
	\textbf{Main algorithm.} If $J=0$, then stop the main algorithm, and output $X_1=A, X_2=B$. In this case, there is no edge between $A$ and $B$, by condition 3 and $|A|,|B|\geq \frac{n}{2}$. By the choice of $\alpha$, this output satisfies condition (ii)'.
	
	Suppose next that $J\geq 1$. For $i=1,\dots,J$, let $V_i$ be the set of vertices $v\in B$ such that $2^{i-1}\leq |N(v)\cap A|<2^{i}$, and let $V_{0}$ be the set of vertices $v\in B$ such that $N(v)\cap A=\emptyset$. Then, by condition 3, we have $B=\bigcup_{i=0}^{J} V_{i}$.
	
	Let $k, \; 1\leq k\leq J$ be the largest integer for which  $t_k<|V_k|$. First, consider the case where there is no such $k$. Then $$n-\sum_{i=J+1}^{J_0}t_{i}-|V_0|\leq n-|B^{*}|-|V_0|=|B|-|V_0|=\sum_{i=1}^{J} |V_i|\leq \sum_{i=1}^{J}t_i,$$
where the first inequality follows from condition 2, and the first equality is the consequence of condition 1. Comparing the left-hand and right-hand sides, and using (\ref{equ1}), we get $|V_{0}|\geq n/2$. In this case, stop the algorithm and output $X_1=V_{0}$ and $X_2=A$. Note that $\alpha(\frac{n}{|X_i|})^{1/2}<2$ is satisfied for $i=1,2$, so this output satisfies condition (ii)'.
	
	Suppose that there exists $k$ with the desired property.  Remove the elements of $V_i$ for $i>k$ from $B$, and add them to $B^{*}$. Then we added at most $ \sum_{i=k+1}^{J} t_i$ elements to $B^{*}$. Setting $J:=k$, properties 1-3 are still satisfied.
\smallskip

	Now we shall run a \emph{sub-algorithm}. Let $Z_0=V_k$. With help of the sub-algorithm, we construct a sequence $Z_0\supset\dots\supset Z_r$ satisfying the following properties. During each step of the sub-algorithm, we either find an output satisfying (i)', or we will move certain elements of $A$ to $A^{*}$. At the end of the $l$-th step of this algorithm, $Z_l$ will be the set of vertices in $B$ that still have at least $2^{k-1}$ neighbours in $A$. We stop the algorithm if $Z_l$ is too small.
	
	\begin{description}
		\item[Sub-algorithm.]    Suppose that $Z_{l}$ has already been defined. If $|Z_{l}|<2t_k$, then let $r=l$, stop the sub-algorithm, remove the elements of $Z_l$ from $B$, and add them to $B^{*}$. Make the update $J:=k-1$, and move to the next step of the main algorithm. Note that $B^{*}$ satisfies condition 2. Later, we will see that all the other properties are satisfied.
		
		On the other hand, if $|Z_l|\geq 2t_k$, we define $Z_{l+1}$ as follows. Let $x_l=\frac{|Z_l|}{t_k}$. Say that a vertex $v\in A$ is \emph{heavy} if
		$$|N(v)\cap Z_l|\geq \frac{x_l 2^k}{t_k}|Z_l|=\left(\frac{|Z_l|}{t_k}\right)^{2}2^k=\frac{|Z_l|^{2}}{n}=:\Delta_l,$$ and let $H_l$ be the set of heavy vertices. Counting the number of edges $f$ between $H_l$ and $Z_l$ in two ways, we can write
		$$|H_l|\Delta_l\leq f< |Z_l|2^{k},$$
		which gives $|H_l|<\frac{t_k}{x_l}$. Remove the elements of $H_l$ from $A$ and add them to $A^{*}$. Examine how the degrees of the vertices in $Z_l$ changed, and consider the following two cases:
		
		\begin{description}
			\item[Case 1.] At least $\frac{|Z_l|}{2}$ vertices in $Z_l$ have at least $2^{k-1}$ neighbors in $A$.
			
			Let $T$ be the set of vertices in $Z_l$ that have at least $2^{k-1}$ neighbors in $A$, so $|T|\geq \frac{|Z_l|}{2}$. Pick each element of $A$ with probability $p=2^{-k}$, and let $S$ be the set of selected vertices. We say that $v\in T$ is \emph{good} if $|N(v)\cap S|=1$, and let $Y$ be the set of good vertices. We have
			$$\mathbb{P}(v\mbox{ is good})=|N(v)\cap A|p(1-p)^{|N(v)\cap A|-1}\geq \frac{1}{2}(1-2^{-k})^{2^{k}}\geq \frac{1}{6},$$ so that $\mathbb{E}(|Y|)\geq \frac{|T|}{6}\geq \frac{|Z_l|}{12}$. Therefore, there exists a choice for $S$ such that $|Y|\geq \frac{|Z_l|}{12}$. Let us fix such an $S$. For each $v\in S$, let $Y_v$ be the set of elements $w\in Y$ such that $N(w)\cap S=\{v\}$. Also, note that $$|Y_v|\leq |N(v)\cap Z_l|\leq \min\{\epsilon n,\Delta_l\}=:\Delta_l'.$$ In other words, the sets $Y_v$ for $v\in S$ partition $Y$ into sets of size at most $\Delta_l'$. Here, we have
			$$\frac{|Y|}{\Delta_l'}\geq \frac{|Z_l|}{12\Delta_l'}\geq \max\left\{\frac{n}{12|Z_{l}|},\frac{|Z_l|}{\epsilon n}\right\}.$$
			By the choice of $\epsilon$, the right-hand side is always at least $6$. But then we can partition $S$ into $t\geq \frac{|Y|}{3\Delta_l'}\geq 2$ parts $W_1,\dots,W_t$ such that the sets $X_i=\bigcup_{v\in W_i} Y_{v}$ have size at least $\Delta_l'$ for $i=1,\dots,t$. The resulting sets $X_1,\dots,X_t$ satisfy that
			$$t\geq  \frac{|Y|}{3\Delta_l'}\geq  \frac{n}{36|Z_{l}|}\geq \frac{1}{36}\left(\frac{n}{\Delta_l}\right)^{1/2}\geq \frac{1}{36}\left(\frac{n}{|X_i|}\right)^{1/2}.$$
			Stop the main algorithm, and output $t$ and the $2t$ disjoint sets $W_1,\dots,W_t$ and $X_1,\dots,X_t$. By the choice of $\alpha$, this output satisfies (i)'.
			
			\item[Case 2.] At most $\frac{|Z_l|}{2}$ vertices in $Z_l$ have at least $2^{k-1}$ neighbours in $A$.
			
			In this case, define $Z_{l+1}$ as the set of elements of $Z_l$ with at least $2^{k-1}$ neighbours in $A$ (then $Z_{l+1}$ is the set of all elements in $B$ with at least $2^{k-1}$ neighbours in $A$ as well). Also, move to the next step of the sub-algorithm.
		\end{description}
		
		We need to check that, if the main algorithm is not terminated, then after the sub-algorithm ends, conditions 1-3 are still satisfied. Conditions 1 and 3 are clearly true, and 2 holds for $B^{*}$. It remains to show that 2 holds for $A^{*}$ as well. Note that, as $|Z_{l+1}|\leq \frac{|Z_{l}|}{2}$ for $l=0,\dots,r-1$, and $|Z_{r-1}|\geq 2t_k$, we have $|Z_l|\geq 2^{r-l}t_k$ and $x_l\geq 2^{r-l}$. Compared to the first step of the sub-algorithm, $|A^{*}|$ increased by
		$$\sum_{l=0}^{r-1}|H_l|\leq \sum_{l=0}^{r-1}\frac{t_k}{x_l}\leq \sum_{l=0}^{r-1} \frac{t_k}{ 2^{r-l}}<t_{k}.$$
		Therefore, condition 2 is also satisfied.
	\end{description}
	
	In every step of the main algorithm, $J$ decreases by at least one, so the main algorithm will stop in a finite number of steps. When the algorithm stops, its output will satisfy either (i)' or (ii)'.
\end{proof}

\section{The proof of Theorem \ref{thm:qEH}}\label{sect:monpath3}
	
	Now we are in a position to prove Theorem \ref{thm:qEH}. Let $G$ be an ordered graph. The \emph{transitive closure} of $G$ is the ordered graph $G'$ on the vertex set $V(G)$ in which $x$ and $y$ are connected by an edge if and only if there exists a monotone path in $G$ with endpoints $x$ and $y$.
	
	\begin{proof}[Proof of Theorem \ref{thm:qEH}]
		 Let $\epsilon_1,\alpha_1$ be the constants given by Lemma \ref{lemma:embedding}. Furthermore, define the following constants: $c_1=\frac{\epsilon_1}{2}$, $c_{i+1}=\frac{\epsilon_{1}c_{i}}{4}$ (for $i=1,2,\dots$), $\epsilon=\frac{c_k}{2}$, and $\alpha=\frac{\alpha_1c_k^{1/2}}{2}$.
\smallskip
		
		  Let $G$ be an ordered graph on $n$ vertices such that
		 \begin{enumerate}
		 	\item the maximum degree of $G$ is at most $\epsilon n$,
		 	\item there exist no $t$ and $t$ disjoint subsets $X_1,\dots,X_t\subset V(G)$ such that $t\geq \alpha(\frac{n}{|X_i|})^{1/2}$ and there is no edge between $X_i$ and $X_j$ for $1\leq i<j\leq t$.
		 \end{enumerate}
		  Then, we show that $G$ contains a monotone path of size $k$ as an induced subgraph. In particular, we find $k$ vertices $x_1\prec\dots\prec x_k$ with the following properties. For $s=1,\dots,k$,
		
		  \begin{description}
		  	\item[(a)] $x_1,\dots,x_{s}$ is an induced monotone path.
		  	\item[(b)] Let $$U_s=V(G)\setminus \left(\bigcup_{i=1}^{s-1} N(x_i)\right),$$
		  	let $G_{s}=G[U_s\cup \{x_s\}]$, and let $G_s'$ be the transitive closure of $G_s$. Then the forward degree of $x_s$ in $G_s'$ is at least $c_s n$.
		  \end{description}
		
		First, we find a vertex $x_1$ with the desired properties, that is, if $G'$ is the transitive closure of $G$, then the forward degree of $x_1$ must be at least $c_1n$. Let $A$ be the set of the first $n/2$ elements of $V(G)$, and set $B=V(G)\setminus A$. Also, let $H$ denote the bipartite subgraph of $G'$ with parts $A$ and $B$. By Lemma \ref{lemma:embedding}, at least one of the following three conditions is satisfied.
	
	    \begin{description}
	    	\item[(i)] There exist $t\geq 2$ and $2t$ disjoint sets $W_1,\dots,W_t\subset A$ and $X_1,\dots,X_t\subset B$ such that $$t\geq \alpha_1\left(\frac{|A|}{|X_i|}\right)^{1/2}=2^{-1/2}\alpha_1\left(\frac{n}{|X_i|}\right)^{1/2}\geq \alpha\left(\frac{n}{|X_i|}\right)^{1/2},$$ and $X_{i}\subset N_{H}(W_{i})$ for $i=1,\dots,t$, but $X_{i}\cap N_{H}(W_j)=\emptyset$ for $i\neq j$.
	    	\item[(ii)] There exist $X_1\subset A$ and $X_2\subset B$ such that
	    	$$2>\alpha_1 \left(\frac{|A|}{|X_i|}\right)^{1/2}=2^{-1/2}\alpha_1\left(\frac{n}{|X_i|}\right)^{1/2}\geq \alpha\left(\frac{n}{|X_i|}\right)^{1/2},$$ and there is no edge between $X_1$ and $X_2$.
	    	\item[(iii)] There exists $v\in A$ such that $|N_{H}(v)|\geq \epsilon_1 |A|=c_1 n.$
	    \end{description}

     	As non-edges of $G'$ are also non-edges of $G$, (ii) cannot hold, by property 2 of $G$ (at the beginning of the proof). Suppose that (i) holds. Note that there is no edge between $X_i$ and $X_j$ in $G$, for $1\leq i<j\leq t$. Suppose for contradiction that $x\in X_i$ and $y\in X_j$ are joined by an edge in $G$, for some $x\prec y$. Then there exists $w\in W_i$ such that $wx\in E(G')$, but $wy\not\in E(G')$. This is a contradiction, as this means that there is a monotone path from $w$ to $x$ in $G$, so there is a monotone path from $w$ to $y$ as well. Hence, there is no edge between $X_i$ and $X_j$ for $1\leq i<j\leq t$, which contradicts 2. Therefore, (iii) must hold: there exists a vertex $x_1\in V(G)$ whose forward degree in $G'=G_1'$ is at least $c_1 n$.
     \smallskip
     	
     	Suppose that we have already found $x_1,\dots,x_{s}$ with the desired properties, for some $1\leq s\leq k-1$. Then we define $x_{s+1}$ as follows. Let $X$ be the forward neighbourhood of $x_{s}$ in $G_s$, let $Y$ be the forward neighbourhood of $x_{s}$ in $G_s'$, and let $Z=Y\setminus X$.  As $|X|\leq \epsilon n$ and $|Y|\geq c_s n$, we have $|Z|\geq \frac{c_s}{2}n$. Let $A$ be the set of the first $\frac{|Z|}{2}$ elements of $Z$ with respect to $\prec$, and let $B=Z\setminus A$. A monotone path in $G_s$ is said to be \emph{good} if none of its vertices, with the possible exception of the first one, belongs to $X$.  For every $v\in A$, there exists at least one element $x\in X$ such that $v\in N^{+}_{G_s'}(x)$; assign the largest (with respect to $\prec$) such element $x$ to $v$. Then there is a good monotone path from $x$ to $v$.
\smallskip

        Define a bipartite graph $H$ between $A$ and $B$ as follows. If $v\in A$ and $y\in B$, and $x\in X$ is the vertex assigned to $v$, then join $v$ and $y$ by an edge if there is a good monotone path from $x$ to $y$. Applying Lemma \ref{lemma:embedding} to $H$, we conclude that at least one of the following three statements is true.

         \begin{description}
        	\item[(i)] There exist $t\geq 2$ and $2t$ disjoint sets $W_1,\dots,W_t\subset A$ and $X_1,\dots,X_t\subset B$ such that $$t\geq \alpha_1\left(\frac{|A|}{|X_i|}\right)^{1/2}>\frac{\alpha_1 c_{s}^{1/2}}{2}\left(\frac{n}{|X_i|}\right)^{1/2}\geq \alpha\left(\frac{n}{|X_i|}\right)^{1/2},$$ and $X_{i}\subset N_{H}(W_{i})$ for $i=1,\dots,t$, but $X_{i}\cap N_{H}(W_j)=\emptyset$ for $i\neq j$.
        	\item[(ii)] There exist $X_1\subset A$ and $X_2\subset B$ such that $$2>\alpha_1 \left(\frac{|A|}{|X_i|}\right)^{1/2}>\frac{\alpha_1 c_{s}^{1/2}}{2}\left(\frac{n}{|X_i|}\right)^{1/2}\geq \alpha\left(\frac{n}{|X_i|}\right)^{1/2},$$ and there is no edge between $X_1$ and $X_2$.
        	\item[(iii)] There exists $v\in A$ such that $|N_{H}(v)|\geq \epsilon_1 |A|=\frac{\epsilon_1 c_{s}}{4}n=c_{s+1}n.$
        \end{description}

        Suppose first that (i) holds. Then, as before, we show that there is no edge between $X_i$ and $X_j$ in $G$ for $1\leq i<j\leq t$. Suppose that $u\in X_i$ and $w\in X_j$ are joined by an edge in $G$, for some $u\prec w$. Then there exists $v\in W_i$ such that $vu\in E(H)$, but $vw\not\in E(H)$. Let $x\in X$ be the vertex assigned to $v$. Then we can find a good monotone path from $x$ to $u$. Since $uw$ is an edge of $G$, there is a good monotone path from $x$ to $w$, contradicting the assumption $vw\not\in E(H)$. Therefore, there cannot be any edge between $X_i$ and $X_j$ in $G$, which means that (i) contradicts 2.

       Suppose next that (ii) holds. Again, we can show that there is no edge between $X_1$ and $X_2$ in $G$, contradicting 2. Suppose that $v\in X_1$ and $y\in X_2$ are joined by an edge in $G$, and let $x\in X$ be the vertex assigned to $v$. There is a good monotone path from $x$ to $v$ in $G_s$, so there is a good monotone path from $x$ to $y$, contradicting the assumption that $vy$ is not an edge of $H$.

       Therefore, we can assume that (iii) holds. Let $v\in A$ be a vertex of degree at least $c_{s+1}n$ in $H$, and let $x_{s+1}\in X$ be the vertex assigned to $v$. We show that $x_{s+1}$ satisfies the desired properties. We have $U_{s+1}=U_s\setminus X$, and the forward degree of $x_{s+1}$ in $G_{s+1}'$ is exactly the number of vertices $y$ such that there is a good monotone path from $x_{s+1}$ to $y$. That is, the forward degree of $x_{s+1}$ is at least $|N_{H}(v)|\geq c_{s+1}n$, as required. This completes the proof.
	\end{proof}

\section{The construction---Proof of Theorem \ref{thm:construction}}\label{constructionsection}

In this section, we present our construction for Theorem \ref{thm:construction}. The construction involves expander graphs, which are defined as follows.

Recall that for any graph $H$ and any $U\subset V(H)$, we denote by $N(U)=N_{H}(U)$ the neighborhood of $U$ in $H$.  The \emph{closed neighborhood} of $U$ is defined as $U\cup N_{H}(U)$, and is denoted by $N[U]=N_H[U]$. The graph $H$ is called an {\em $(n,d,\lambda)$-expander} if $H$ is a $d$-regular graph on $n$ vertices, and for every $U\subs V$ satisfying $|U|\le |V|/2$, we have $|N_H[U]|\ge (1+\lambda)|U|$. By a well-known result of Bollob\'as \cite{B88}, a random 3-regular graph on $n$ vertices is a $(n,3,\lambda_{0})$-expander with high probability for some absolute constant $\lambda_{0}>0$. In the rest of this section, we fix such a constant $\lambda_{0}$. For explicit constructions of expander graphs see, e.g., \cite{M94}.

For any positive integer $r$, let $H^{r}$ denote the graph with vertex set $V(H)$ in which two vertices are joined by an edge if there exists a path of length at most $r$ between them in $H$. Here we allow loops, so that in $H^r$ every vertex is joined to itself. We need the following simple property of expander graphs.

 \begin{claim}\label{claim:expander}
 	Let $H$ be an $(n,d,\lambda)$-expander graph and let $r\ge 1$. For any subsets $X,Y\subs V(H)$ such that there is no edge between $X$ and $Y$ in $H^{r}$, we have $|X||Y|\leq n^{2}(1+\lambda)^{-r}$.
 \end{claim}

 \begin{proof}
 	Let $X_i=N_{H^i}[X]$ and $Y_i=N_{H^i}[Y]$ for $i=0,1,\dots,r$. It follows from the definition of expanders that, if $|X_{i}|\leq \frac{n}{2}$, then $$|X|\leq \frac{1}{2}n(1+\lambda)^{-i}.$$ Similarly, if $|Y_{i}|\leq \frac{n}{2}$, then $|Y|\leq \frac{1}{2}n(1+\lambda)^{-i}$. If $X$ and $Y$ are not connected by any edge in $H^{r}$, then $X_i$ and $Y_{r-i}$ must be disjoint for every $i$. Let $\ell$ be the largest number in $\{0,1,\dots,r\}$ such that $|X_{\ell}|\le n/2$.
 	
 	If $\ell=r$, then $|X|<n(1+\lambda)^{-r}$, and hence $|X||Y|\le n^2(1+\lambda)^{-r}$.
 	
 	If $\ell<r$, then $|X_{\ell+1}|>n/2$ and $|Y_{r-\ell-1}|\le n/2$. Therefore, we have ${|Y|\le  n(1+\lambda)^{-(r-\ell-1)}}$.  Using the inequality $1+\lambda\leq 2$, we obtain
 	$$|X||Y| \le \frac{1}{4}n^2(1+\lambda)^{-r+1}\leq n^2(1+\lambda)^{-r}.$$
 \end{proof}

\begin{claim}\label{claim:maxdegexpander}
	For any $d$-regular graph $H$ and $r\ge 1$, we have $\Delta(H^{r})\leq (d+1)^{r}$.
\end{claim}

\begin{proof}
  	Trivial, by induction on $r$.
\end{proof}

Our construction is based on the following key lemma.

\begin{lemma}\label{lemma:construction}
	 Let $k,m,f$ be positive integers. Let $A_{1},\dots,A_{k}$ be disjoint sets of size $m$, and
suppose that there exists an $(m,3,\lambda_{0})$-expander.

Then there is a graph $G$ on the vertex set $V=\bigcup_{i=1}^{k}A_{i}$ such that
	\begin{enumerate}
		\item $\Delta(G)\leq 4^{f2^{k}}$,
		\item there are no three vertices $x,y,z\in V$ such that $x\in A_{a}, y\in A_{b}, z\in A_{c}$ for some $a<b<c$, and $xy,xz\in E(G)$, but $yz\not\in E(G)$,
		\item for any $a\neq b$ and any pair of subsets $X\subset A_{a}$ and $Y\subset A_{b}$ not connected by any edge of $G$, we have $|X||Y|\leq {m^{2}}{(1+\lambda_{0})^{-f}}$.
	\end{enumerate}
\end{lemma}

\begin{proof}
	Let $H$ be an $(m,3,\lambda_{0})$-expander. Let $\phi:V\rightarrow V(H)$ be an arbitrary function such that $\phi$ is a bijection when restricted to the set $A_{i}$, for $i=1,\dots,k$. Define the graph $G$, as follows. Suppose that $x\in A_{a}$ and $y\in A_{b}$ for some $a<b$. Join $x$ and $y$ by an edge if there exists a path of length at most $f2^{a-1}$ between $\phi(x)$ and $\phi(y)$ in $H$. By Claim \ref{claim:maxdegexpander}, the maximum degree of $G$ is at most $\sum_{i=1}^{k-1}4^{f2^{i}}\leq 4^{f2^{k}}$, so that $G$ has property 1.
	
	To see that $G$ also has property 2, consider $x\in A_{a},\; y\in A_{b},\; z\in A_{c}$ such that $a<b<c$ and $xy, xz\in E(G)$. We have to show that $yz\in E(G)$. By definition, there exists a path of length at most $f2^{a-1}$ between $\phi(x)$ and $\phi(y)$ in $H$, and there exists a path of length at most $f2^{a-1}$ between $\phi(x)$ and $\phi(z)$. But then there exists a path of length at most $f2^{a}\leq f2^{b-1}$ between $\phi(y)$ and $\phi(z)$, so $yz$ is also an edge of $G$.
	
	It remains to verify that $G$ has property 3. If $1\leq a< b\leq k$ and $X\subset A_{a}$ and $Y\subset A_{b}$ are not connected by any edge in $G$, then there is no edge between $\phi(X)$ and $\phi(Y)$ in $H^{f2^{a-1}}$. By Claim \ref{claim:expander}, we have $|X||Y|\leq m^{2}(1+\lambda_{0})^{-f2^{a-1}}\leq m^{2}(1+\lambda_{0})^{-f}.$
	\end{proof}

Now we are in a position to prove Theorem \ref{thm:construction}.

\begin{proof}[Proof of Theorem \ref{thm:construction}]
	Let $k= \frac{2}{\epsilon}$, $f= \frac{\log_{2} n}{4\cdot 2^{k}}$, and $m=\frac{n}{k}$. We show that the theorem holds with $\delta=\frac{\log_{2}(1+\lambda_{0})}{2^{k}}$.
	
	Let $A_{1},\dots,A_{k}$ be disjoint sets of size $m$. By Lemma \ref{lemma:construction}, there exists a graph $G_0$ on $V=\bigcup_{i=1}^{m}A_{i}$ satisfying conditions 1-3 with the above parameters.

	Define the ordered graph $G$ on the vertex set $V$ as follows. Let $\prec$ be any ordering on $V$ satisfying $A_{1}\prec\dots \prec A_{k}$. For any $x\in A_{a}$ and $y\in A_{b}$, join $x$ and $y$ by an edge of $G$ if either $a\neq b$ and $xy\in E(G_0)$, or $a=b$. Then the maximum degree of $G$ is at most $\frac{n}{k}+\Delta(G_0)\leq \epsilon n$. Notice that the complement of $G$ does not contain a bi-clique of size $n^{1-\delta}$. Indeed, if $(X,Y)$ is a bi-clique in $\overline{G}$, then there exists $a\neq b$ such that $|X\cap A_{a}|\geq \frac{|X|}{k}$ and $|Y\cap A_{b}|\geq \frac{|Y|}{k}=\frac{|X|}{k}$. Thus, $\frac{|X|^{2}}{k^{2}}\leq \frac{m^{2}}{n^{2\delta}}$, which implies that $|X|\leq n^{1-\delta}$.
	
	It remains to show that $G$ contains neither $S$, nor $P$ as an induced ordered subgraph. Let us start with $S$. Suppose that there are four vertices, $v_{0}\prec v_{1}\prec v_{2}\prec v_{3}$, in $G$ such that $v_{0}v_{1},v_{0}v_{2},v_{0}v_{3}\in E(G)$, but $v_{1}v_{2},v_{2}v_{3},v_{1}v_{3}\not\in E(G)$. Let $v_{0}\in A_{a}$, $v_{1}\in A_{b}$, $v_{2}\in A_{c}$, and $v_{3}\in A_{d}$, then $a\leq b\leq c\leq d$. If $c=a$, then $b=a$, which implies $v_{1}v_{2}\in E(G)$, contradiction. Therefore, $a<c\leq d$. As $v_{2}v_{3}\not\in E(G)$, we must have $c<d$ as well. But then the three vertices $v_{0},v_{2},v_{3}$ contradict property 2, so that $G$ does not contain $S$ an induced ordered subgraph.
	
	To show that $G$ does not contain $P$, we can proceed in a similar manner. Suppose for contradiction that there are four vertices, $v_{0}\prec v_{1}\prec v_{2}\prec v_{3}$, in $G$ such that $v_{0}v_{2},v_{0}v_{3},v_{1}v_{2}\in E(G)$, but $v_{0}v_{1},v_{1}v_{3},v_{2}v_{3}\not\in E(G)$. Let $v_{0}\in A_{a}$, $v_{1}\in A_{b}$, $v_{2}\in A_{c},$ and $v_{3}\in A_{d}$, where $a\leq b\leq c\leq d$. We have $a<b$, otherwise $v_{0}v_{1}\in E(G)$. In the same way, $c<d$, otherwise $v_{2}v_{3}\in E(G)$. Therefore, $a<c<d$, and the vertices, $v_{0},v_{2},$ and $v_{3}$, contradict condition 2 of Lemma \ref{lemma:construction}.
\end{proof}

\end{document}